\documentclass[12pt]{article}
\usepackage{geometry}
\usepackage{graphicx}
\usepackage{amsmath, amssymb, amsthm, xcolor}
\usepackage{caption}
\usepackage[hidelinks]{hyperref}

\title{Circuit Decompositions for Triangulations of Surfaces}
\author{Jens Harlander, Maizie Quatrone}
\date{September 2026}

\newtheorem{theorem}{Theorem}
\newtheorem{lemma}{Lemma}
\newtheorem{corollary}{Corollary}
\newtheorem{proposition}{Proposition}

\theoremstyle{definition}
\newtheorem{definition}{Definition}

\begin{document}

\maketitle

\begin{abstract}
An Euler circuit of a graph is a closed path that visits every edge of the graph exactly once. Euler circuit and circuit decomposition problems can also be formulated for higher dimensional simplicial complexes. 
An Euler \textit{$k$-circuit} in $K$ is a cyclic sequence of vertices $v_1\dots v_n$ such that every $k+1$ adjacent terms $\{ v_i,v_{i+1},\dots,v_{i+k} \}$ (indexed modulo $n$) form a $k$-simplex, and every $k$-simplex of $K$ appears exactly once in the sequence $v_1v_2\ldots v_n(v_1v_2\ldots v_k)$. 
We investigate the 2-circuit decomposition problem for triangulated closed compact surfaces. Using interior angles of paths we define an obstruction $\theta(S)\in H^1(S,\mathbb Z_3)$, where $S$ is an orientable surface, which vanishes if and only if $S$ has a 2-circuit decomposition. We also show that a non-orientable triangulated surface has a 2-circuit decomposition if and only if its orientable 2-fold cover does.

\end{abstract}

\section{Introduction}

Arguably, the K\"onigsberg bridge problem \cite{wikipedia} marks the beginning of combinatorial topology. In graph theoretic terms it ask under what conditions a connected graph permits a closed edge path  that uses every edge of the graph exactly once. This problem was solved by Euler in 1736: The graph admits such a closed path (these days referred to as an Euler circuit) if and only if the valency at every vertex is even. There is a natural generalization of the K\"onigsberg bridge problem to higher dimensional complexes which, in general, is not well understood. We will describe it after introducing some notation.

We assume that the reader is familiar with the idea of a simplicial complex, its topological realization, fundamental group, homology, and homotopy. For definitions see Hatcher \cite{Hatcher} and Munkres \cite{Munkres}.

\begin{definition}
    Let $K$ be a simplicial complex of dimension $\ge k$. A \textit{$k$-circuit} in $K$ is a cyclic sequence of vertices $\vec v=v_1\dots v_n$ such that every $k+1$ adjacent terms $\{ v_i,v_{i+1},\dots,v_{i+k} \}$ (indexed modulo $n$) form a $k$-simplex in $K$ which appears exactly once in the sequence $v_1v_2\ldots v_n(v_1v_2\ldots v_k)$. A \textit{k-circuit decomposition} of a simplicial complex $K$ is a collection of $k$-circuits such that each k-simplex of $K$ occurs in exactly one of these circuits. If a k-circuit decomposition consists of a single circuit, then we call such a circuit an \textit{Euler k-circuit}.
\end{definition}

\noindent{\bf K\"onigsberg bridge question in higher dimensions}: Under what conditions does a simplicial complex admit an Euler k-circuit or at least a k-circuit decomposition? 

If a connected graph has a 1-circuit decomposition, then the circuits can be fused to produce an Euler 1-circuit. Thus a connected graph has a 1-circuit decomposition if and only if it admits an Euler 1-circuit. This is not true in higher dimensions. See Figure \ref{no_euler_figure}.

\begin{figure}[h]
    \centering
    \includegraphics[width=0.5\linewidth]{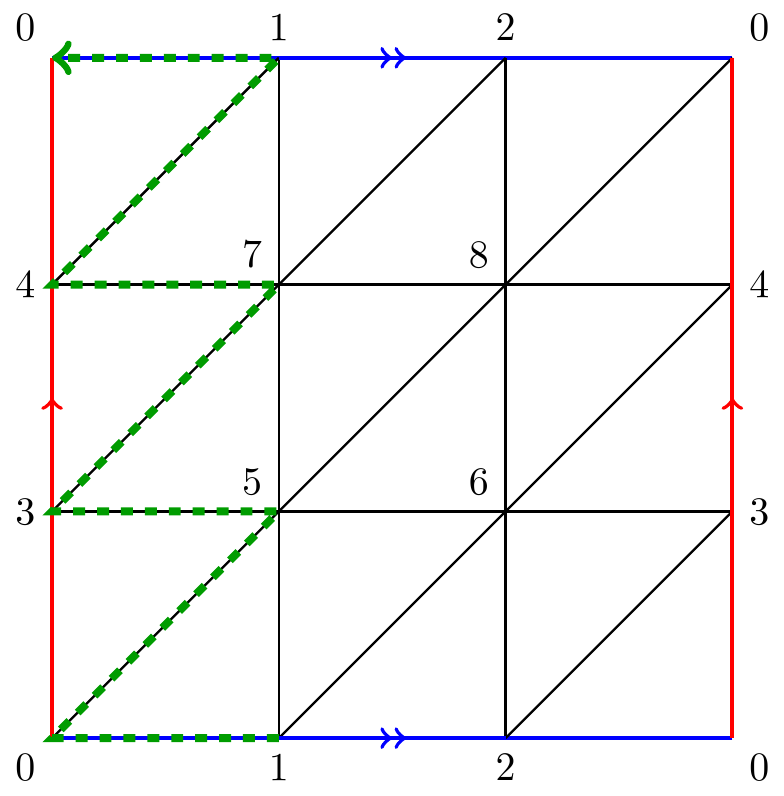}
    \caption{A 2-circuit decomposition of a torus: 105374(10), 216587(21), 023648(02); it can be shown (left to the reader) that no Euler 2-circuit exists for this triangulation.}
    \label{no_euler_figure}
\end{figure}

The above question was considered, although in set theoretic language, by Chung, Diaconis and Graham \cite{Chung} in the case where $K=\Delta^n$ is a single $n$-simplex. They posed a conjecture concerning the existence of Euler $k$-circuits in this setting, which was confirmed by Jackson \cite{Jackson} for small $k$. 
The full conjecture was confirmed by Glock, Joos, K\"uhn, and Osthus \cite{Glock}.

If $K$ is a simplicial complex and $v$ is a vertex then the {\em $k$-valency at $v$}, $\text{val}_k(v)$, is the number of $k$-simplices containing $v$. We say $K$ is {\em $m$-divisible} in dimension $k$ if $\text{val}_k(v)$ is divisible by $m$ for every vertex. 

\begin{proposition} 
    If $K$ has a $k$-circuit decomposition then it is $k+1$-divisible.
\end{proposition}

\begin{proof}
    Let $v$ be a vertex of $K$. If $v$ does not occur in any of the circuits, then $v$ is not contained in any $k$-simplex, so $\text{val}_k(v)=0$, which is divisible by $k+1$. Assume that $v$ occurs in some circuit as $v=v_i$. Then $v$ is contained in the simplices $\{ v_i, v_{i+1}, \ldots, v_{i+k} \}$, $\{ v_{i-1}, v_{i}, \ldots, v_{i-1+k} \}$, \dots, $\{ v_{i-k}, v_{i-1}, \ldots, v_{i} \}$, a set of $k+1$ simplices. Thus, if $v$ occurs $m$ times in the circuits, then val$_k(v)=m(k+1)$.
\end{proof}

\begin{figure}[h!]\label{subdivision_figure}
    \centering
    \includegraphics[width=0.8\linewidth]{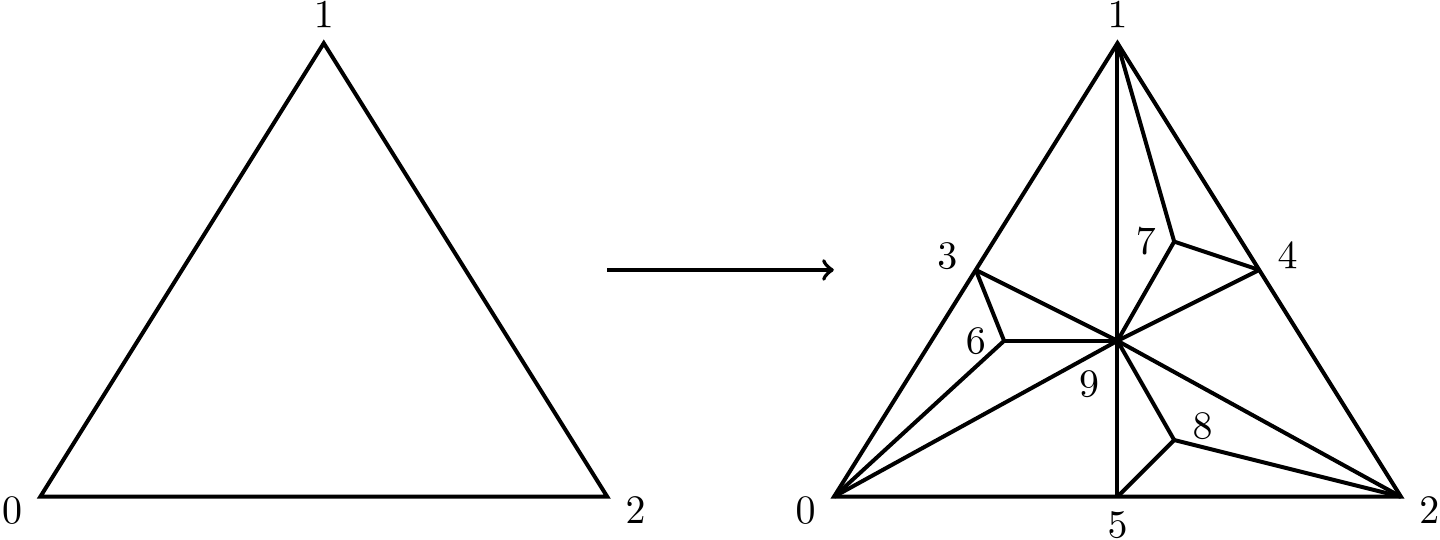}
    \caption{Subdividing triangles in a 2-complex leads to a triangulation that has a 2-circuit decomposition. This is because the subdivided triangle admits the Euler 2-circuit 590639174928(59).}
    \label{subdivision_figure}
\end{figure}

In this paper we will answer the circuit decomposition problem in dimension 2 for triangulated closed compact surfaces $S$. We will construct a cohomological obstruction $\theta(\bar S) \in H^1(\bar S,\mathbb Z_3)$, where $\bar S$ is the 2-fold orientable cover of $S$ in case $S$ is not orientable, and $\bar S=S$ in case it is, and show:
\begin{theorem}
    Let $S$ be a 3-divisible closed and compact surface. Then $S$ has a 2-circuit decomposition if and only if $\theta(\bar S)=0$.
\end{theorem}

\section{Fundamental group, orientations, and angles}\label{fundamentalgroup_section}

A path in $K$ is a sequence of oriented edges $p=e_1\ldots e_n$, where $e_i=[v_iv_{i+1}]$, $1\le i\le n$. It is closed if $v_{n+1}=v_1$. We denote by $P(K,v_0)$ the set of closed paths in $K$ starting and ending at $v_0$, including the empty path. If $e=[uv]$ is an oriented edge, then we denote by $\bar e=[vu]$ the edge with opposite orientation. Pairs $e\bar e$ and $\bar e e$ are called a \textit{canceling edge pair}. We define two operations on $P(K,v_0)$:
\begin{enumerate}
    \item Insertion and deletion of canceling pairs in paths in $P(K,v_0)$;
    \item Rerouting across triangles: If $p=e_1\ldots e_n$ is a path in $P(K,v_0)$, $e_j=[uv]$ and $\{u,v,w\}$ is a 2-simplex, then we may replace $e_j$ by the path $[uw][wv]$; furthermore if $e_je_{j+1}$ is of the form $[uw][wv]$, and $\{u,v,w\}$ is a 2-simplex, then the pair may be replaced by the single edge $[uv]$.
\end{enumerate} 

We say that two elements of $P(K,v_0)$ are \textit{homotopic} if one can be transformed into the other by the two operations. Homotopy defines an equivalence relation on $P(K,v_0)$. If $p \in P(K,v_0)$, we denote by $[p]$ its homotopy class. We define
$$\pi_1(K,v_0)=\{ [p]\ |\ p\in P(K,v_0)\}.$$ 
If $p=e_1\ldots e_m$ and $p'=e'_1\ldots e'_n$ are elements of $P(K,v_0)$, then we denote by $pp'=e_1\ldots e_me'_1\ldots e'_n$ the concatenation. One can show that $[p]\cdot [p']=[pp']$ defines a group operation on $\pi_1(K,v_0)$. It follows from Van Kampen's theorem (see \cite{Hatcher}) that our $\pi_1(K,v_0)$ is isomorphic to the fundamental group of the topological realization of $K$.

If $K$ is a simplicial complex and $v$ is one of its vertices, then the link at $v$, lk$(v)$, is the subcomplex consisting of simplices $\tau$ so that $v\not\in \tau$ and $\tau\cup\{ v \}$ is a simplex of $K$. The star at $v$, st$(v)$, is the cone on  lk$(v)$. That is, it consists of $v$, the simplices of the link, together with the simplices of the form  $\tau\cup \{ v\}$, where $\tau$ is a simplex in the link.

Let $S$ be a triangulation of a closed compact surface (from now on simply referred to as surface). If $v$ is a vertex in $S$, then st$(v)$ is a triangulated disc and lk$(v)$ a subdivided circle (the boundary of the disc). A \textit{local orientation at a vertex $v$} is a choice of orientations on the 2-simplices of st$(v)$ so that the boundary of the 2-cycle $\gamma_v=\sum_{\sigma\in\text{st}^{(2)}(v)}\sigma$ is the 1-cycle consisting of exactly the edges of lk$(v)$. This closed path around $v$ best conveys the idea of a local orientation. 

An orientation of $S$ is a choice of orientations on its 2-simplices $\sigma\in S^{(2)}$ such that $\gamma=\sum_{\sigma\in S^{(2)}}\sigma$ is a 2-cycle. It can be shown that if $S$ admits an orientation, then $\gamma$ generates the second homology $H_2(S)$. Also note that an orientation of $S$ induces local orientations at every vertex.

Assume that $S$ is an oriented surface (i.e. a surface with a fixed orientation). Let $v$ be a vertex of $S$ and let $e_1e_2$ be a path in st$(v)$, say $e_1=[u,v]$ and $e_2=[v,w]$. We define the \textit{angle} $\theta(e_1e_2)$ to be the number of edges in the lk$(v)$ encountered when traveling from $w$ to $u$ in clockwise direction. If $p=e_1e_2\ldots e_n$ is a closed path then we define the angle sum $\theta(p)=\sum_{i=1}^n \theta(e_ie_{i+1})$, where $e_{n+1}=e_1$. If $p$ and $p'$ are in $P(S,v_0)$, then $\theta(pp')=\theta(p)+\theta(p')$.

\begin{theorem}
    If $S$ is an oriented 3-divisible surface, then 
    $$\theta\colon \pi_1(S,v_0)\to \mathbb Z_3,$$ defined by $\theta[p]=\theta(p)\mod3$, is a well defined group homomorphism.
\end{theorem}

\begin{proof} 
    Let $p\in P(S,v_0)$. We have to check that $\theta(p)$ does not change when we apply the homotopy moves 1 and 2 defined at the beginning of this section. Suppose $e_1e_2$ is a subsequence that occurs in $p$. Insert a canceling pair $e_1e\bar ee_2$. We have to check that $\theta(e_1e_2)\equiv\theta(e_1e)+\theta(e\bar e)+\theta(\bar e e_2)\pmod3$. Note that $\theta(e\bar e)=0$, which leaves us with showing $\theta(e_1e_2)\equiv\theta(e_1e)+\theta(\bar e e_2)\pmod3$. There are two cases to consider. Let $v$ be the terminal vertex of $e_1$. First assume that when listing the edges of st$(v)$ that contain $v$ using the orientation at $v$, the edge $e$ appears after $e_1$ but before $e_2$. See Figure \ref{canceling_figure}.
    Then
    $$\theta(e_1e)+\theta(\bar e e_2)=\text{val}(v)+\theta(e_1e_2)\equiv \theta(e_1e_2) \pmod3$$
    \begin{figure}[h!]
        \centering
        \includegraphics[width=0.4\linewidth]{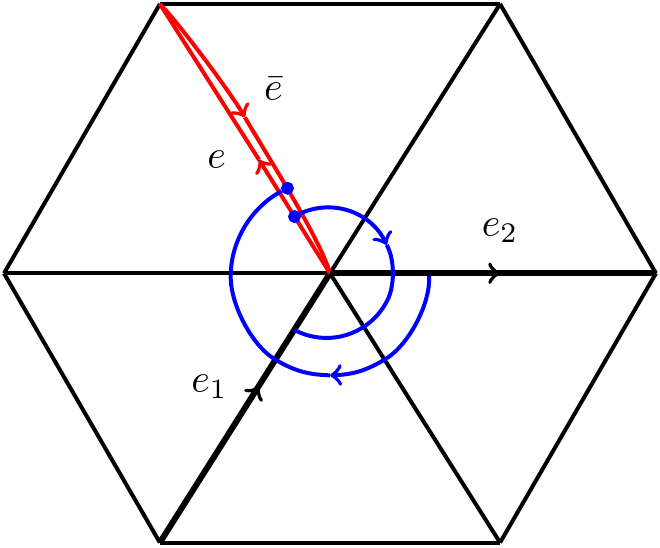}
        \caption{This figure shows the insertion of a canceling pair.}
        \label{canceling_figure}
    \end{figure}
    because $S$ is assumed to be 3-divisible, so val$(v)\equiv0\pmod3$. 
    
    \noindent The other case where the edge $e$ appears after $e_2$ is treated in the same way, we leave the details to the reader. This shows that insertion of canceling pairs into the path $p$ does not alter $\theta(p)$.
    
    Nest assume that $e_1e_2e_3$ is a subsequence that occurs in $p$ and we replace $e_2$ by $e_2'e_2''$, where $e_2$, $e_2'$, and $e_2''$ are edges of a single triangle. Let $v$ be the terminal vertex of $e_2'$. We consider the case depicted in Figure \ref{rerouting_figure}. 
    We have 
    \begin{align*}
        \theta(e_1e_2'e_2''e_3) &=\theta(e_1e_2')+\theta(e_2'e_2'')+\theta(e_2''e_3) \\
        &=\theta(e_1e_2)-1+\text{val}(v)-1+\theta(e_2e_3)-1 \\
        &\equiv\theta(e_1e_2)+\theta(e_2e_3)\pmod3
    \end{align*}
    because val$(v)\equiv0\pmod3$. The other case, where the triangle sits on the other side of $e_2$ is treated similarily, and we leave the details to the reader. This shows that rerouting a path across a triangle does not change $\theta(p)$.
    
    \begin{figure}[h]
        \centering
        \includegraphics[width=0.3\linewidth]{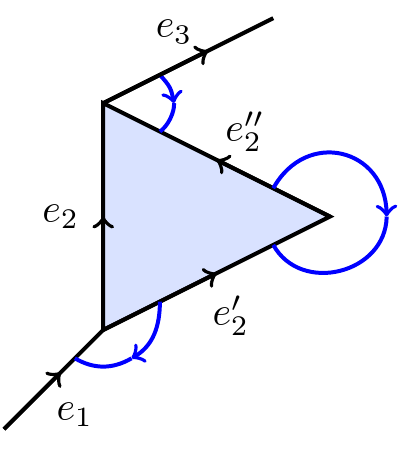}
        \caption{Rerouting across a triangle. Note that we assume clockwise orientations at all vertices, something we can guarantee only in the orientable case.}
        \label{rerouting_figure}
    \end{figure}
    
    We have shown that $\theta[p]=\theta(p)\mod3$ is well defined. Now notice that 
    $$\theta[pp']\equiv\theta(pp')\equiv\theta(p)+\theta(p')\equiv\theta[p]+\theta[p']\pmod3.$$
    This shows that $\theta$ is a group homomorphism.
\end{proof}

We point out that since $\theta$ maps to an abelian group it factors over the commutator subgroup and we obtain a group homomorphism
$$\theta\colon H_1(K)\to \mathbb Z_3.$$

\section{Labelings}

\begin{definition}\label{j_labelable_def}(surface labeling) 
    Suppose $S$ is a 3-divisible oriented surface. A \textit{labeling} on $S$ is a function 
    $$\lambda\colon E(S)\to \mathbb Z_3=\{0, 1, 2 \}$$
    such that for every vertex $v$ of $S$ the following labeling condition holds: The consecutive edges of st$(v)$ that contain $v$ can be listed $e_0,\ldots e_l$ in orientation direction so that  $\lambda(e_i)=i$ (mod 3). 
\end{definition}

We also need a weaker labeling notion that does not require a global orientation.

\begin{definition}\label{weaklabel_def}(weak surface labeling) 
    Let $S$ be a 3-divisible surface. A \textit{weak labeling} on $S$ is a function 
    $$\lambda\colon E(S)\to \mathbb Z_3=\{0, 1, 2 \}$$ which takes on only the values $0$ and $1$
    such that for every vertex $v$ of $S$ the following weak labeling condition holds: For any chosen orientation at $v$ the consecutive edges of st$(v)$ containing $v$ can be listed $e_0,\ldots e_l$ in orientation direction so that  $\lambda(e_i)=0$ (mod 3) if and only if $i=0$ (mod 3).
\end{definition}

Note that a labeling becomes a weak labeling if every value 2 is replaced by the 1.
A 3-divisible surface may or may not have a weak labeling. Note that every 2-simplex in a labeled surface carries the labels $0$, $1$, and $2$ on its sides ($0$, $1$, $1$ in case of a weak labeling).
Also, knowing the value $\lambda(e)$ of a single edge in a 3-divisible orientable surface $S$ completely determines the labeling. For suppose that $e=\{v,w\}$. It is clear that the labeling condition determines the labeling of all edges containing $v$ and all edges containing $w$. Since the 1-skeleton of $S$ is connected, we see that all labels are indeed determined.

\begin{theorem}\label{j_strong_theorem}
    Suppose $S$ is a 3-divisible, oriented, weakly labeled surface. Then $S$ admits a labeling.
\end{theorem}

\begin{proof}
    Let $\lambda'$ be a weak labeling of $S$ and let $v$ be a vertex. Any orientation at $v$, in particular the one induced by the orientation of $S$, gives us a list of consecutive edges $e_0,\ldots e_l$ of st$(v)$ with weak labels $0, 1, 1, 0,\ldots, 0, 1, 1$. We now change the list of weak labels $0, 1, 1, 0,\ldots, 0, 1, 1$ at the vertex $v$ by the labels $0, 1, 2, 0,\ldots, 0, 1, 2$, using the local orientation induced by the orientation of $S$, and obtain a local labeling $\lambda_v$ at $v$. We do this at all vertices. We claim that if $e=\{ u, v \}$, then $\lambda_u(e)=\lambda_v(e)$. Note first that if $\lambda'(e)=0$, then $\lambda_u(e)=\lambda_v(e)=0$, because we never change $0$ labels in the process. Suppose $\lambda_u(e)=1$ and $\lambda_v(e)=2$. See Figure \ref{localglobal_figure}. This implies that at the vertex $w$ we have two consecutive edges in st$(w)$ labeled with $0$, which contradicts the fact that $\lambda'$ is a weak labeling. Given an edge $e$ in $S$ we define $\lambda(e)=\lambda_x(e)$, where $x$ is a vertex of $e$, to obtain a labeling of $S$.
\end{proof}

\begin{figure}[h]
    \centering
    \includegraphics[width=0.5\linewidth]{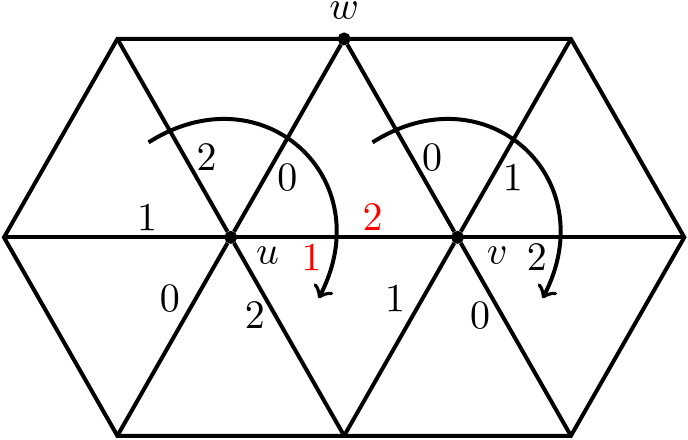}
    \caption{$\lambda_u(e)\ne \lambda_v(e)$ contradicts the fact that $\lambda'$ is a weak labeling.}
    \label{localglobal_figure}
\end{figure}

Note that the proof of Theorem \ref{j_strong_theorem} gives a bijection between weak labelings and labelings on a 3-divisible oriented surface: Given a weak labeling the construction given in the proof produces a unique labeling, and replacing all values $2$ by $1$ turns a labeling into a weak labeling.

\begin{lemma}\label{j_numberlabel_lemma}
    Let $S$ be a 3-divisible, labelable, oriented surface. Then $S$ has exactly three distinct labelings (using the given orientation).
\end{lemma}

\begin{proof}
    Suppose $\lambda_0$ is a labeling on $S$. Then $\lambda_1(e)\equiv\lambda_0(e)+1\pmod3$ and $\lambda_2(e)\equiv\lambda_0(e)+2\pmod3$ define labelings as well. Thus $S$ has at least three distinct labelings. Suppose $\lambda$ is a labeling (using the fixed orientation).  Let $v$ be a vertex. Then at $v$ there are exactly three different labelings at $v$, and each is a restriction to the st$(v)$ of one of the $\lambda_i$. But the labeling at $v$ determines the rest of the labeling. Thus $\lambda=\lambda_i$, for some $i=0, 1, 2$. 
\end{proof}

It can also be shown that a 3-divisible non-orientable weakly labelable surface has a unique weak labeling. \\

If $p$ is a path in $S$ then we denote by $N(p)=\bigcup \text{st}(v)$, where the union is taken over the vertices on the path.

\begin{definition}\label{j_labelableloop_def}(closed path labeling) 
    Suppose $p$ is a closed path in a 3-divisible oriented surface $S$. A \textit{labeling} of $p$ is a function
    $$\lambda\colon E(N(p))\to \mathbb Z_3$$
    so that the labeling condition of Definition \ref{j_labelable_def} holds for every vertex on the path. Similarly, suppose $p$ is a path in a 3-divisible surface $S$. Then the labeling is called \textit{weak} if it only takes on the values $0$ and $1$ and the weak labeling condition of Definition \ref{weaklabel_def} holds for every vertex on the path.
\end{definition}

\begin{lemma}\label{j_pathlabel_lemma}
    Let $S$ be a 3-divisible surface. Then $S$ has a weak labeling if and only if every closed path $p$ has a weak labeling.
\end{lemma}

\begin{proof}
    A labeling of $S$ restricts to a labeling of a given closed path $p$. For the other direction, simply note that there is a closed path $p$ in $S$ that passes through every vertex. Thus, a labeling of $p$ gives a labeling of $S$.
\end{proof}

\begin{lemma}\label{j_construct_cycle_lemma}
    Let $S$ be a 3-divisible surface with weak labeling $\lambda$. Let $\sigma$ be one of its triangles. Then $S$ carries a unique (up to reversal) 2-circuit $v_0v_1\ldots v_n$ containing $\sigma$ such that $\lambda\{ v_i, v_{i+1}\}\ne 0$ for all $i$. 
\end{lemma}

\begin{proof}
     Let $\sigma=\sigma_0=\{v_0,v_1,v_2\}$. Note that every triangle contains exactly two edges labeled 1. Assume WLOG that $e_0=\{ v_0,v_1\}$ and $e_1=\{v_1,v_2\}$ are edges labeled 1. We start our circuit as $v_0v_1v_2$. There is exactly one other triangle $\sigma_1=\{v_1, v_2, v_3\}$ in the surface $S$ which contains the edge $e_1$. Note that the edge $e_2=\{ v_2, v_3\}$ is labeled 1 because otherwise two edges in the star of $v_2$ that are not three edges apart would be labeled $0$. We extend our circuit to $v_0v_1v_2v_3$. We continue in  this manner until we encounter for the first time a triangle that we have seen before, say $\sigma_n=\sigma_j$. We have constructed
     $$v_0v_1v_2\ldots v_{n+1}$$
     so that the $e_i=\{v_i, v_{i+1}\}$, $0\le i\le n-1$ are distinct edges and $\sigma_j=\{v_j, v_{j+1}, v_{j+2}\}$, $0\le j\le n-1$, are distinct triangles.

    The edge $e_n=\{ v_n, v_{n+1}\}$ is contained in the distinct triangles $\sigma_{n-1}$ and some $\sigma_j$. See Figure \ref{toruscircuit_figure}. If $j\ne 0$ then since all other edges have been used, this implies that $\{ v_n, v_{n+1}\}=\{v_j, v_{j+2}\}$, which is impossible because the label on $\{v_j, v_{j+2}\}$ is 0 and all edges $e_0, e_1,\ldots, e_n$ are labeled 1. Thus $j=0$ and $[v_nv_{n+1}]=[v_0v_{1}]$ or $[v_nv_{n+1}]=[v_1v_0]$. Suppose the latter holds. In that case the lk$(v_0)$ contains two edges that are labeled zero but are only \textcolor{purple}{two} edges apart, which can not occur in a labeling. 
    Thus $[v_nv_{n+1}]=[v_0v_{1}]$. It then follows that $v_0v_1\ldots v_n$ is a 2-circuit containing $\sigma$ such that $\lambda\{ v_i, v_{i+1}\}=1$.
    
    \begin{figure}[h]
        \centering
        \includegraphics[width=0.35\linewidth]{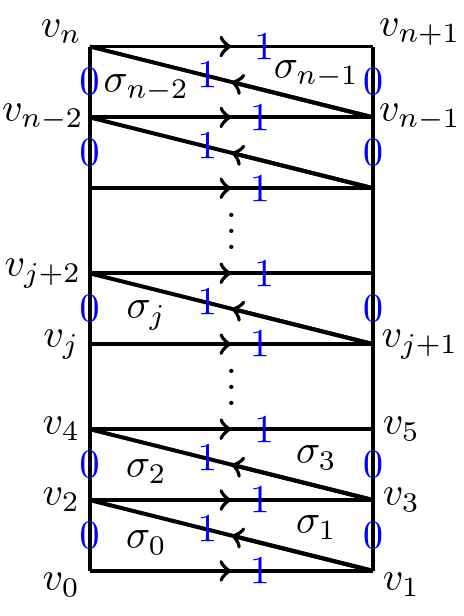}
        \caption{}
        \label{toruscircuit_figure}
    \end{figure}
    
    Let us address uniqueness (up to reversal). To start out construction we could have chosen different orientations on the edges $e_0$ and $e_1$ and a different starting point. We could have started with $v_2v_1v_0$ instead of $v_0v_1v_2$. There is exactly one other triangle $\sigma'_1=\{v_1, v_0, v'_3\}$ in the surface $S$ which contains the edge $e_0$ and it is $\sigma'_1=\{v_1, v_0, v_n\}$. Thus the construction of our circuit continues as $v_2v_1v_0v_n$. We are building the reversed circuit $v_nv_{n-1}\ldots v_0$. This shows that the circuit that contains $\sigma$ and has all subsequent edges labeled non-zero is indeed unique up to reversal.
\end{proof}

In the proof of the above lemma a gallery of triangles is produced that form an annulus (in case the number of triangles are even), or a M\"obius band (in case the number of triangles is odd). In fact it is not difficult to see that there is bijective correspondence between such annuli/M\"obius bands and 2-circuits.\\
    
Here is one of our main results.

\begin{theorem}\label{j_decomp_theorem}
    A 3-divisible surface $S$ admits a 2-circuit decomposition if and only if it admits a weak labeling. 
\end{theorem}

\begin{proof} 
    Suppose $S$ admits a 2-circuit decomposition. Since every triangle occurs exactly once in one of the circuits, every edge occurs as $[v_iv_{i+1}]$ or as $[v_iv_{i+2}]$ somewhere in the circuits. Note that $[v_iv_{i+1}]$ occurs exactly once, because the edge determines the two triangles $\{v_i,v_{i+1},v_{i+2}\}$ and $\{v_{i-1},v_i,v_{i+1}\}$ in the surface $S$ containing it. However $[v_iv_{i+2}]$ determines only one triangle $\{v_i,v_{i+1},v_{i+2}\}$, so it must occur exactly twice in the circuits. This also shows that an edge cannot both occur in the form $[v_iv_{i+1}]$ and in the form $[v_iv_{i+2}]$. If $e$ is an edge occurring in the form $[v_iv_{i+1}]$ we define $\lambda(e)=1$, otherwise $\lambda(e)=0$. We have to check that this satisfies the weak labeling condition.
    
    Let $v$ be a vertex and assume it occurs as $v_i$ in some circuit. This contributes the three triangles $\{ v_{i-2}, v_{i-1}, v_i\}$, $\{ v_{i-1}, v_{i}, v_{i+1}\}$, $\{ v_{i}, v_{i+1}, v_{i+2}\}$ to st$(v)$. The edges $\{ v_{i-2}, v_{i}\}$, $\{ v_{i-1}, v_{i}\}$, $\{ v_{i}, v_{i+1}\}$, $\{ v_{i}, v_{i+2}\}$ carry labels $0$, $1$, $1$, $0$, respectively. Every time $v$ occurs somewhere in a circuit we add three triangles to the st$(v)$ with edges labeled zero exactly three edges apart. This shows that our $\lambda$ is indeed a weak labeling.
        
    Conversely, suppose $S$ has a weak labeling. Let $\sigma_0$ be a triangle in $S$. By Lemma \ref{j_construct_cycle_lemma}, there is a unique (up to reversal) circuit $\vec v_0$ containing $\sigma_0$ with subsequent edges labeled non-zero. If $\vec v_0$ contains all triangles of $S$, we are done. If not, let $\sigma_1$ be a triangle not contained in $\vec v_0$. By Lemma \ref{j_construct_cycle_lemma}, there is a circuit $\vec v_1$ containing $\sigma_1$ with subsequent edges labeled non-zero which is disjoint from $\vec v_0$ by uniqueness. We continue in this fashion to produce a circuit decomposition. 
\end{proof}

\section{A cohomological obstruction}

In Section \ref{fundamentalgroup_section} we defined a homomorphism $\theta\colon \pi_1(K,v_0)\to \mathbb Z_3$ that maps a homotopy class $[p]$ to its angle sum $\theta(p)$ modulo 3. It factors through the commutator subgroup and induces a homomorphism on the first homology $\theta\colon H_1(K)\to \mathbb Z_3$ which we also, by a slight abuse of notation, call $\theta$. The universal coefficient theorem for cohomology provides an isomorphism $H^1(K,\mathbb Z_3)\to \text{Hom}(K,\mathbb Z_3)$, and so we can interpret $\theta$ as an element of $H^1(K,\mathbb Z_3)$. It turns out that in the case of an orientable surface, $\theta$ is an obstruction for 2-circuit decomposability.

\begin{lemma}\label{j_cohompathobstruction_lemma}
    Let $S$ be a 3-divisible oriented surface and let $p\in P(S,v_0)$ be a closed path based at $v_0$. Then $p$ is labelable if and only if $\theta(p)\equiv0\pmod3$.
\end{lemma}

\begin{proof}
    Assume first that $p=e_1\ldots e_n$ is labelable. Let $\lambda$ be a labeling. We make the observation that $$\lambda(e_{i+1})\equiv\lambda(e_{i})+\theta(e_ie_{i+1})\pmod3.$$
    This implies
    \begin{align*}
        \theta(p) &= \theta(e_1e_2)+\theta(e_2e_3)+\ldots +\theta(e_ne_1) \\
        &\equiv (\lambda(e_2)-\lambda(e_1))+(\lambda(e_3)-\lambda(e_2))+\ldots +(\lambda(e_1)-\lambda(e_n)) \\
        &\equiv0\pmod3
    \end{align*}
    
    For the converse assume that $\theta(p)\equiv0\pmod3$. Let $\lambda(e_1)=0$ and define recursively $\lambda(e_{i+1})=\lambda(e_{i})+\theta(e_ie_{i+1})\mod3$. Since $e_{n+1}=e_1$ it suffices to check that $\lambda(e_{n+1})=\lambda(e_1)=0$. Our recursive definition gives
    $$\lambda(e_{n+1})\equiv\lambda(e_1)+\theta(e_1e_2)+\ldots +\theta(e_ne_1)\equiv\lambda(e_1)+\theta(p)\equiv0\pmod3.$$
    Since $S$ is 3-divisible and oriented we can uniquely extend the labeling of the edges of $p$ to a labeling of the edges of $N(p)$ that contain vertices of $p$.
\end{proof}

\begin{theorem}\label{j_cohom_theorem}
    Let $S$ be a 3-divisible oriented surface. Then $\theta=0$ if and only if $S$ has a circuit decomposition.
\end{theorem}

\begin{proof}
    Assume first that $\theta=0$ is the zero map. Then $\theta[p]\equiv\theta(p)\equiv0\pmod3$ for every $p$. It follows from Lemma \ref{j_cohompathobstruction_lemma} that every $p$ is labelable. Lemma \ref{j_pathlabel_lemma} now implies that $S$ is labelable. Theorem \ref{j_decomp_theorem} now implies that $S$ has a 2-circuit decomposition.
    
    For the other direction assume that $S$ has a 2-circuit decomposition. Then, by Theorem \ref{j_decomp_theorem}, $S$ is labelable. Then every closed path in $S$ is labelable and it follows from Lemma \ref{j_cohompathobstruction_lemma} that $\theta(p)\equiv0\pmod3$ for every $p$. Thus $\theta=0$.
\end{proof}

Theorem \ref{j_cohom_theorem} provides an easy way to construct triangulated surfaces that do or do not have circuit decompositions. See Figure \ref{covering_figure}.

\begin{corollary}\label{j_sphere_corollary}
    A 3-divisible 2-sphere always admits a cycle decomposition.
\end{corollary}
\begin{proof}
    $H_1(S)=0$ and thus $\theta$ is the zero map.
\end{proof}

\section{The non-orientable case}

Let $S$ be a surface and let $pr\colon\bar S\to S$ be a covering space of finite index. Note that since $pr$ restricts to a simplicial isomorphism on stars, $S$ is 3-divisible if and only $\bar S$ is.

\begin{lemma}\label{j_deck_trafo_lemma}
    Let $\bar\lambda$ be a weak labeling on the 3-divisible surface $\bar S$ and $d\colon \bar S\to \bar S$ be a deck transformation. Then $\bar \lambda\circ d$ is also a weak labeling on $\bar S$. Thus, the deck transformation group acts on the set of weak labelings on $\bar S$
\end{lemma}

\begin{proof}
    Let $\bar\lambda$ be a weak labeling on $\bar S$. Let $d$ be a deck transformation. We have to check that $\bar\lambda\circ d$ satisfies the weak labeling condition at every vertex. Let $v$ be a vertex and let $v'=d(v)$. There is a listing of the consecutive edges $e_0', e_1', e_2', e_3' \ldots$ of st$(v')$, using a chosen orientation at $v'$, such that $\bar\lambda(e_0')=0, \bar\lambda(e_1')=1, \bar\lambda(e_2')=1, \bar\lambda(e_3')=0 \ldots$. Since $d$ restricts to a simplicial isomorphism from st$(v)$ to st$(v')$, there is a listing of the consecutive edges $e_0, e_1, e_2, e_3 \ldots$ of st$(v)$, using some orientation at $v$, such that $d(e_i)=e_i'$. We now have $\bar\lambda\circ d(e_i)=\bar\lambda(e_i')$ and obtain
    $$\bar\lambda\circ d(e_0)=0, \bar\lambda\circ d(e_1)=1, \bar\lambda\circ d(e_2)=1, \bar\lambda\circ d(e_3)=0, \ldots$$ 
    The weak labeling condition holds at $v$.
\end{proof}

\begin{theorem}\label{j_double_theorem}
    Let $S$ be a non-orientable 3-divisible surface and let $pr\colon\tilde S\to S$ be its orientable double cover. Then $S$ admits a circuit decomposition if and only if $\tilde S$ does.
\end{theorem}
\begin{proof}
    Suppose $S$ admits a circuit decomposition. By Theorem \ref{j_decomp_theorem} it follows that $S$ admits a weak labeling $\lambda$. Define $\tilde\lambda=\lambda\circ pr$. We will show that $\tilde\lambda$ is a weak labeling on $\tilde S$. Let $\tilde v$ be a vertex of $\tilde S$ and let $v=pr(\tilde v)$. There is a listing $e_0, e_1, e_2, e_3\dots $ of the consecutive edges in st$(v)$ containing $v$ so that $\lambda(e_0)=0, \lambda(e_1)=1, \lambda(e_2)=1, \lambda(e_3)=0 \ldots $. Now note that $pr\colon \text{st}(\tilde v)\to \text{st}(v)$ is a simplicial isomorphism. Thus there is a listing $\tilde e_0, \tilde e_1, \tilde e_2, \tilde e_3\ldots $ of the consecutive edges in st$(\tilde v)$ containing $\tilde v$, where $pr(\tilde e_i)=e_i$. We now have 
    $$\tilde \lambda(\tilde e_0)=0, \tilde\lambda(\tilde e_1)=1, \tilde\lambda(\tilde e_2)=1, \tilde\lambda(\tilde e_3)=0, \ldots $$
    which shows that the weak labeling condition holds at $\tilde v$.
    Thus $\tilde S$ admits a weak labeling and hence, by Theorem \ref{j_decomp_theorem},  $\tilde S$ admits a circuit decomposition.

    Now suppose $\tilde S$ admits a circuit decomposition and hence, by Theorem \ref{j_decomp_theorem}, $\tilde S$ has a weak labeling. Fix an orientation of $\tilde S$.
    By Lemma \ref{j_deck_trafo_lemma}, the deck transformation group $\mathbb Z_2$ acts on the set of weak labelings. It follows from the remark after Theorem \ref{j_strong_theorem} and Lemma \ref{j_numberlabel_lemma} that there are only three weak labelings on $\tilde S$. Thus $\mathbb Z_2$ acts on a set of three elements, and hence must fix one of them. Thus there is a labeling $\tilde \lambda$ which is invariant under all deck transformations. We define $\lambda$ on $S$ in the following way: $\lambda(e)=0$ if $\tilde \lambda(\tilde e)=0$, and $\lambda(e)=1$ if $\tilde \lambda(\tilde e)=1$, where $\tilde e\in pr^{-1}(e)$. Again, using the fact that $pr\colon \text{st}(\tilde v)\to \text{st}(v)$ is a simplicial isomorphism, it follows that $\lambda$ is a weak labeling on $S$. By Theorem \ref{j_decomp_theorem} it follows $S$ has a circuit decomposition.
\end{proof}

We point out that there do exist covering spaces $pr\colon \bar S\to S$ where $\bar S$ is labelable but does not admit a labeling that is invariant under deck transformations. See Figure \ref{covering_figure}.

\begin{figure}[h]
    \centering
    \includegraphics[width=0.75\linewidth]{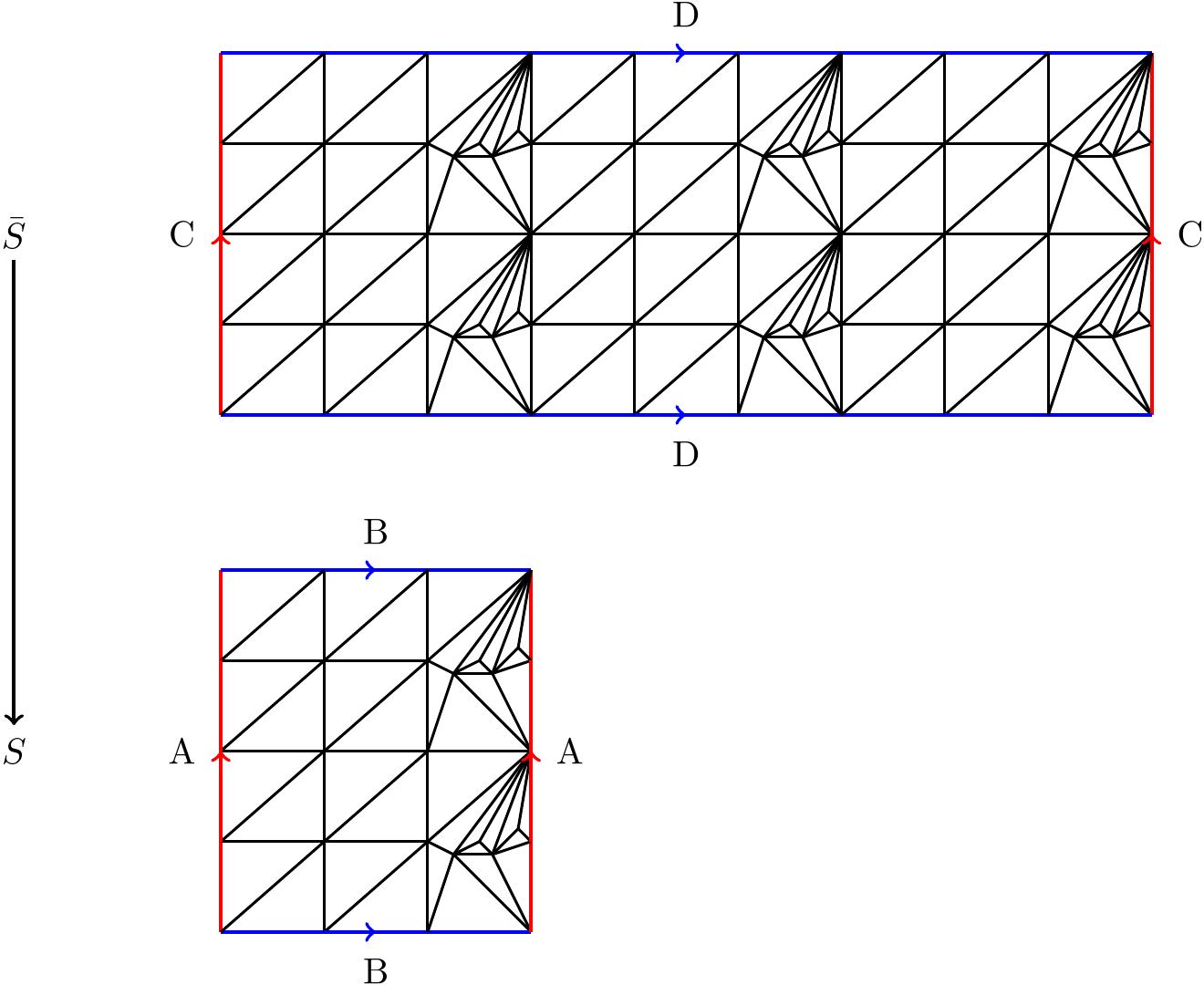}
    \caption{A covering $pr\colon \bar S\to S$ of index 3 of a torus by a torus. Theorem \ref{j_cohom_theorem} implies that $\bar S$ has a circuit decomposition, but $S$ does not:
    Since $\theta(B)=19\equiv1\pmod3$ which is not zero we see that $\theta$ is not zero on S and therefore S does not have a circuit decomposition by Theorem 5. However $\theta(C)=12\equiv0$ and $\theta(D)=3\cdot19\equiv0\pmod3$. Since $[C]$ and $[D]$ generate the fundamental group of $\bar S$, it follows that $\theta$ is zero on $\bar S$, which shows that $\bar S$ has a circuit decomposition.}
    \label{covering_figure}
\end{figure}

\begin{corollary}
    A 3-divisible projective plane $P$ always admits a cycle-decomposition.
\end{corollary}
\begin{proof}
    Let $pr:S^2\to P$ be the orientable double cover of $P$. Since $P$ is 3-divisible, so is $S^2$. The result now follows from Corollary \ref{j_sphere_corollary} and Theorem \ref{j_double_theorem}.
\end{proof}

Finally, our main result Theorem 1 follows directly from Theorem \ref{j_cohom_theorem} and Theorem \ref{j_double_theorem}.

\vspace{1cm}

Jens Harlander, jensharlander@boisestate.edu\\

Maizie Quatrone, maiziequatrone@gmail.com


\begin{thebibliography}{99}

    \bibitem{Chung} F. Chung, P. Diaconis, R. Graham, {\em Universal cycles for combinatorial structures}, Discrete Mathematics 110 (1992), 43-59.
    
    \bibitem{Glock} S. Glock, F. Joos, D. Kühn, D. Osthus, {\em Euler Tours in Hypergraphs}, Combinatorica 40 (2020), 679-690.
    
    \bibitem{Hatcher} A. Hatcher, Algebraic Topology, Cambridge University Press, Cambridge, 2002. 
    
    \bibitem{Jackson} B. Jackson, {\em Universal cycles of k-subsets and k-permutations}, Discrete Mathematics, 117
    (1993) 141-150.
    
    \bibitem{Munkres} J. R. Munkres, Elements of Algebraic Topology, CRC Press, 1996.
    
    \bibitem{wikipedia} Wikipedia contributors, {\em Seven Bridges of K\"onigsberg}, Wikipedia, The Free Encyclopedia, accessed August 2026. \href{https://en.wikipedia.org/w/index.php?title=Seven_Bridges_of_K%C3%B6nigsberg&oldid=1371189868}{https://en.wikipedia.org/wiki/Seven\_Bridges\_of\_K\"onigsberg}
\end{thebibliography}
\end{document}